\documentclass[12pt,a4paper]{amsart}
\usepackage{amssymb}
\usepackage{graphicx,amssymb,amsfonts,epsfig,amsthm,a4,amsmath,url}
\usepackage[latin1]{inputenc}

\newtheorem{ThmIntro}{Theorem}
\newtheorem{CorIntro}[ThmIntro]{Corollary}
\newtheorem{PropIntro}[ThmIntro]{Proposition}

\newtheorem{thm}{Theorem}[section]
\newtheorem{cor}[thm]{Corollary}
\newtheorem{lem}[thm]{Lemma}

\newtheorem{prop}[thm]{Proposition}
\theoremstyle{definition}
\newtheorem{defn}[thm]{Definition}

\theoremstyle{remark}
\newtheorem{rem}[thm]{Remark}
\newtheorem{ex}[thm]{Example}

\numberwithin{equation}{section}

\begin{document}

\title[Contracting automorphisms and $L^p$-cohomology]{Contracting automorphisms and $L^p$-cohomology in degree one}
\author{Yves Cornulier, Romain Tessera}%
\date{January 10, 2010}




\maketitle

\begin{abstract}
We characterize those Lie groups, and algebraic groups over a local field of
characteristic zero, whose first reduced $L^p$-cohomology is zero
for all $p>1$, extending a result of Pansu. As an application, we obtain a description
of Gromov-hyperbolic groups among those groups. In particular we prove that any non-elementary Gromov-hyperbolic algebraic group over a non-Archimedean local field of zero characteristic is quasi-isometric to a 3-regular tree. We also extend the study to general semidirect products of a locally compact group by a cyclic group acting by contracting automorphisms. 
\end{abstract}

\section{Introduction}

Let $G$ be a locally compact group, endowed with a {\em left} Haar
measure. Let $\rho=\rho_G$ denote the {\it right} regular representation
of $G$ on the space $\mathbf{R}^G$ of all real-valued functions on $G$,
defined by $(\rho(g) f)(h)=f(hg)$ for $g,h\in G, f\in \mathbf{R}^G$. 
Given $p\in
[1,+\infty[$, let $D^p(G)$ denote the set of $p$-Dirichlet functions on $G$, namely
measurable functions $f$
on $G$ such that $b(g)=f-\rho(g) f$ belongs to $L^p(G)$ for all $g\in G$, and
such that $g\mapsto b(g)$ is continuous from $G$ to $L^p(G)$. This space contains
constant functions and $L^p$
functions (indeed, the continuity of $g\mapsto b(g)$ is clear for continuous compactly
supported functions, which form a dense subspace of $L^p(G)$). 
For every compact subset $Q\subset G$, we define a seminorm on $
D^p(G)$ by
$$\|f\|_{D^p,Q}=\sup_{g\in Q}\|f-\rho(g)f\|_p.$$
We equip $D^p(G)$ with the topology induced from the seminorms $\|\cdot\|_{D^p,Q}$ for all $Q$. Then it follows from \cite[Lemma 5.2]{T1} that $D^p(G)/\mathbf{R}$ is isomorphic, as a topological vector space with a $G$-action, to the space of 1-cocycles $Z^1(G,\rho_G^p)$, where $\rho_G^p$ denotes the right regular representation on $L^p(G)$.

Set $$H_p^1(G)=D^p(G)/(L^p(G)+\mathbf{R})\simeq H^1(G,\rho_G^p)$$ and
$$\overline{H_p^1}(G)=D^p(G)/\overline{(L^p(G)+\mathbf{R})}\simeq \overline{H^1}(G,\rho_G^p).$$  The first $L^p$-cohomology and the reduced one coincide if and
only if the norms $\|f\|_p +\|f\|_{D^p,Q}$ and $\|f\|_{D^p,Q}$ are
equivalent on $L^p(G)$ for some compact subset $Q$. If $G$ is non-compact, this
happens if and only if $G$ satisfies the Sobolev inequality
$$\|f\|_p\leq C\|f\|_{D^p,Q},$$ i.e.~if and only if
\cite[Proposition~11.9]{T3} $G$ is either non-unimodular or
non-amenable. This can be reformulated as 
\begin{itemize}
\item If $G$ is non-compact, amenable and unimodular, then $H^1_p(G)$ is non-Hausdorff (and in particular is non-zero);
\item Otherwise, $H^1_p(G)=\overline{H^1_p}(G)$.
\end{itemize}

A definition of the first $L^p$-cohomology of a locally compact group in the context of metric measured spaces is given in \cite{Pan95} (see also \cite[Section 3]{T1} and Appendix \ref{qii}); the equivalence between the two definitions is obtained in \cite[Section 5]{T1}.

\begin{defn}\label{dco}~
Throughout the paper, {\it Lie groups} refer to real Lie groups.
\begin{itemize}
\item The unit component of a locally compact group $G$ is denoted by $G_0$.
\item Let $\mathbf{Z}$ act by automorphisms on a locally compact group $H$. A {\it vacuum subset} is a subset $U\subset H$ such that 
for every compact subset $M\subset H$, we have
$k\cdot M\subset U$ for $k$ large enough.
We say that $\mathbf{Z}$ {\it contracts} $H$
if there exists a compact vacuum subset $U$. The action is {\it strictly contracting} if any neighborhood of 1 is a vacuum subset (in some papers this is simply referred to as {\it contracting}).
\item We say that a Lie group $G$ with finitely many connected components is of {\it Heintze
type} if it is isomorphic to a
semidirect product $S\ltimes N$, where $N$ is a non-compact simply connected
nilpotent Lie group, $S$ contains $\mathbf{Z}$ as a cocompact
subgroup which contracts $N$.
\item We say that a locally compact group $G$ is of {\it non-Archimedean Heintze type} if it is isomorphic to a
semidirect product $S\ltimes N$, where $N$ is a 
non-compact, totally disconnected, locally compact nilpotent group, $S$ contains $\mathbf{Z}$ as a cocompact
subgroup which contracts $N$.
\item We say that a locally compact group $G$ is of {\it rank-one type} (resp.~{\it non-Archimedean rank-one type}) if for some (unique) compact normal subgroup $W$ in $G$, some finite index subgroup of $G/W$ is isomorphic to the quotient of a simple group of rank one over the reals (resp.~over some non-Archimedean local field) by its center.
\end{itemize}
\end{defn}

\begin{ex}~
\begin{itemize}
\item Let $G=(K\times\mathbf{R})\ltimes N$ be a Lie group with finitely many connected components, with $K$ compact, and assume that for every positive $t\in\mathbf{R}\subset G$, every eigenvalue $\lambda$ of $\textnormal{Ad}(t)$ acting on the Lie algebra $\mathfrak{n}$ of $N$ satisfies $|\lambda|<1$. Then $N$ is nilpotent and simply connected, $G$ is of Heintze type, and conversely every Lie group with finitely many connected components of Heintze type is of this form. Besides, Heintze's main result in \cite{H} is that a connected Lie group of dimension $\ge 2$ has a left-invariant Riemannian metric of negative curvature if and only if it is a simply connected solvable Heintze Lie group.

\item Similarly, if $G=(K\times\mathbf{K}^*)\ltimes N$ is an algebraic group over a non-Archimedean local field with $K$ compact, and for every $t\in \mathbf{K}^*$ with 
$|t|>1$, all eigenvalues $\lambda$ of $\textnormal{Ad}(t)$ acting on $\mathfrak{n}$ satisfy $|\lambda|<1$. Then $G$ is of non-Archimedean Heintze type. When $\mathbf{K}$ has characteristic zero, every connected linear algebraic $\mathbf{K}$-group of non-Archimedean Heintze type, is of this form.\end{itemize}
\end{ex}

Our main result is the following theorem. We include the statement (1) in order to give a complete picture, but only prove the other ones.

\begin{ThmIntro}\label{main}
Consider $p\in [1,+\infty)$. Let $G$ be a connected Lie group, or a linear algebraic group over a non-Archimedean local field of characteristic zero.
\begin{itemize}
\item[(1)]\cite{T1,T2} If $G$ is amenable and unimodular, then $\overline{H_p^1}(G)=0$ for all $p>1$.
\item[(2)]Suppose that $G$ is Heintze or rank-one type (Lie or non-Archimedean). Then $\overline{H_p^1}(G)\neq 0$ for $p$ large enough.
\item[(3)] Otherwise, $\overline{H_p^1}(G)=0$ for all $p\ge 1$.
\end{itemize}
\end{ThmIntro}

Pansu \cite{Pan06} obtains (2) and (3) when $G$ is a connected solvable Lie group. His approach is based on a definition of $L^p$-cohomology for Riemannian manifolds, involving differentiation. This gives rise to some technical issues, related to the fact that the gradient of an $L^p$-function need not be~$L^p$.
The equivalence between the two definitions of $L^p$-cohomology is given in \cite[Chap.~4-5]{T1}. The discrete version, given here, allows a unified proof of (2) and (3) in the Lie and non-Archimedean case. Theorem \ref{mix} below provides a more general and more precise statement, which contains (2) as a particular case. We first focus on some corollaries on Theorem \ref{main}.

\begin{CorIntro}Equivalences:
\renewcommand{\labelenumi}{(\ref{c2}.\arabic{enumi})}
\begin{enumerate}\item\label{2a} For some $p>1$, we have $\overline{H^1_p}(G)\neq 0$;
\item $G$ is of Heintze or rank-one type (Lie or non-Archimedean).
\end{enumerate}\label{c2}\end{CorIntro}

Lie groups admitting a left invariant Riemannian metric with negative sectional curvature have been described algebraically in \cite{H}. An application of our main result is to provide a characterization of Gromov-hyperbolicity for Lie groups and algebraic groups over a local field of zero characteristic. This can be understood as a ``large-scale" version of Heintze's theorem \cite{H}. Note that one advantage of focusing on the large-scale geometry of a Lie group is that the statements are independent on a choice of Riemannian metric and also make sense when we consider the word length with respect to a compact generating set. In general, a locally compact group is said to be {\it Gromov-hyperbolic} if it is compactly generated and, viewed as a metric space with the word metric to some compact generating set, is a Gromov-hyperbolic metric space.

By \cite[Theorem~1.2]{CT}, a connected Lie
group (resp. an algebraic group over a local field with
characteristic zero) with exponential growth has a bi-Lipschitz
embedded 3-regular tree. It then follows from \cite[Theorem~6]{T1} that if such a group is Gromov-hyperbolic, it has
non-trivial first reduced $L^p$-cohomology for large enough $p$. Accordingly, we get the following corollary. 

\begin{CorIntro}\label{ThmIntro:Hyp}Equivalences:
\renewcommand{\labelenumi}{(\ref{c3}.\arabic{enumi})}
\begin{enumerate}
\item\label{3a} The group $G$ is non-elementary Gromov-hyperbolic; 
\item\label{3b} the group $G$ is of Heintze or rank-one type (Lie or non-Archimedean).
\end{enumerate}
\label{c3}\end{CorIntro}


If we distinguish the Lie and the non-Archimedean cases, we get the two following additional corollaries

\begin{CorIntro}\label{c4}
If $G$ is a Lie group with finitely many connected components, we have equivalences:
\begin{enumerate}
\renewcommand{\labelenumi}{(\ref{c4}.\arabic{enumi})}
\item\label{4a} $G$ is [non-elementary] Gromov-hyperbolic;
\item\label{4b} $G$ is quasi-isometric to a simply connected homogeneous manifold of negative curvature [of dimension at least two];
\item\label{4c} $G$ acts properly transitively by isometries on a simply connected homogeneous manifold of negative curvature [of dimension at least two].
\end{enumerate}
\end{CorIntro}

The construction of a left-invariant metric on a simply connected Heintze Lie group is one of the principal results in \cite{H}. 
We need slightly more to obtain the third statement of the corollary for an arbitrary Heintze Lie group, but it turns out that what we need follows from Heintze's construction, namely \cite[Theorem~2]{H}. This amounts to prove that if $G$ is a simply connected Heintze group and $K$ a compact group of automorphisms of $G$, then $G$ possesses a left-invariant Riemannian metric which is $K$-invariant.

\begin{CorIntro}\label{c5}
If $G$ is an algebraic group over a
non-Archimedean local field of characteristic 0, we have equivalences:
\begin{enumerate}
\renewcommand{\labelenumi}{(\ref{c5}.\arabic{enumi})}
\item\label{5a} $G$ is [non-elementary] Gromov-hyperbolic;
\item\label{5b} $G$ is quasi-isometric to a regular tree of finite degree [of degree at least three];
\item\label{5c} $G$ acts properly, cocompactly (i.e.~with finitely many orbits) by isometries on some regular tree of finite degree [of degree at least three].
\end{enumerate}
\end{CorIntro}

Note that all regular trees of degree $3\le d<\infty$ are quasi-isometric to each other. In order to get the third condition, we make use of the following general proposition.

\begin{PropIntro}\label{tree}
Consider a contracting action $\sigma$ of $\mathbf{Z}$ on some
non-compact, totally disconnected, locally compact group $H$.  There exists a proper length function $\ell$ on the semidirect product $G=\mathbf{Z}\ltimes H$ such that the pseudo-metric space $(G,d)$, where  $d(g,g')=\ell(g^{-1}g')$, is isometric to the vertex set of a $r$-regular tree for some $r\geq 3$.
 \end{PropIntro}
 
Turning back to Theorem \ref{main}(2), we have the following more general result.

\begin{ThmIntro}\label{mix}
Let $H$ be a locally compact group whose unit component $H_0$ is not compact, with an action of of a locally compact group $S$, which is contracting in restriction to some cocompact subgroup $\mathbf{Z}=\langle\xi\rangle$ of $S$, and set $G=\mathbf{Z}\ltimes H$. Let $\xi^{-1}$ multiply the Haar measure of $H$ by $\delta>1$, and let $\lambda>1$ be the smallest modulus of eigenvalues greater than one of $\xi^{-1}$ on $H_0$. Set $$p_0=p_0(G)=\log(\delta)/\log(\lambda)\ge 1.$$ Then for all $p\ge 1$,
$$\overline{H_p^1}(G)\neq 0 \textnormal{ if and only if }p>p_0.$$
\end{ThmIntro}

Theorem \ref{mix} is proved by Pansu \cite{Pan90,Pan06} when $H$ is a simply connected solvable Lie group. Theorem \ref{mix} applies to ``mixed" groups such as $\mathbf{Z}\ltimes_{(t^{-1},\ell)}(\mathbf{R}\times\mathbf{Q}_\ell)$ ($t>1$, $\ell$ prime), for which $p_0=\log(t\ell)/\log(t)$. Unlike the case of Lie groups, this provides examples where $p_0$ is (any number) in $]1,2[$. Besides, the assumption in Theorem \ref{mix} of $H_0$ to be non-compact is no restriction, since otherwise when $H_0$ is compact (but not $H$), Proposition \ref{tree} applies and $\overline{H_p^1}(G)\neq 0$ for all $p\ge 1$ (see Proposition \ref{ct} for a direct proof).

\setcounter{tocdepth}{1}
\tableofcontents

In Section \ref{ModularFunctionSection}, we collect some results relating some properties of $L^p$-cocycles to the modular function of the group. These observations, crucial in the non-unimodular case, are largely adapted from \cite{Pan06}. At the additional cost of some structural results on Heintze Lie groups, we then obtain Theorem~\ref{main}(3).

In Section \ref{MainThmSection}, we prove Theorem \ref{mix} and in particular deduce (2) of Theorem \ref{main}.

In Section \ref{hyp}, we complete the proof of the corollaries of Theorem \ref{main}. In particular, we prove the existence of a $G$-invariant metric of negative curvature on $G/K$ for general Heintze Lie groups, and prove Proposition \ref{tree}.

Finally, in Appendix \ref{a1}, we give a general structure theorem for locally compact groups with a contraction, and in Appendix \ref{qii}, we prove the quasi-isometric invariance of $L^p$-cohomology in degree one, an particular case of an unpublished result of Pansu~\cite{Pan95}.

\section{Vanishing of the first $L^p$-cohomology and the modular function}\label{ModularFunctionSection}

In all this section, $p\in[1,\infty)$. The aim of this section is to prove Theorem~\ref{main}(3).

\subsection{Generalities}

Recall that $\rho$ denotes the {\it right} regular representation of $G$ on $L^p(G)$. Let $\Delta$ be the modular function on $G$. For every measurable
function $f$ on $G$ and $g\in G$, we have
$\|\rho(g)f\|=\Delta(g)^{-1/p}\|f\|$, where $\|\cdot\|$ is always
assumed to denote the $L^p$-norm. For every $\xi\in G$, define
$$W_\xi=\{h\in G|(\xi^{-n}h\xi^n)_{n\ge 0}\text{ is bounded}\}.$$
This is a subgroup of $G$.

\begin{lem}
Fix $\xi\in G$ such that $\Delta(\xi)>1$ (such $\xi$ exists if and
only
if $G$ is non-unimodular). Suppose that $u\in D^p(G)$. Then there exists
$u_\infty\in D^p(G)$ such that $u-u_\infty\in L^p(G)$ and
$\rho(\xi)u_\infty=u_\infty$.\label{lem_xinv}\end{lem}

\begin{proof} We have
$$\|\rho(\xi^{n+1})
u-\rho(\xi^n) u\|=\Delta(\xi)^{-n/p}\|\rho(\xi) u-u\|.$$

Therefore the sequence $(\rho(\xi^n) u-u)$ converges in $L^p(G)$
to some function
$v$; we set $u_\infty=v+u$. In particular, the sequence
$(\rho(\xi^n) u)$ converges almost surely to $u_\infty$. In
particular, $\rho(\xi)u_\infty=u_\infty$ almost everywhere.
Moreover, $u-u_\infty\in L^p(G)$, so that $u_\infty\in D^p(G)$ and
defines the same class as $u$ in $H_p^1(G)$. \end{proof}

\begin{lem}
Let $G$ be any locally compact group, and suppose that $u\in
D^p(G)$ satisfies $\rho(H)u=u$ for some non-compact closed
subgroup $H$ of $G$. Let $V$ be the centralizer of $H$ in $G$.
Then $\rho(V)u=u$.\label{lem:centre}\end{lem}

\begin{proof} For any measurable function $f$ on $G$ and any measurable
subset $X$ of $G$, denote by $\|f\|_X$ the $L^p$-norm of $f.1_X$.

Fix any compact subset $X$ of $G$. As $H$ is non-compact, its
subset $H_{+}=H\cap\{\Delta\le 1\}$ is non-compact. As the right
action of $G$ on itself is proper, there exists a sequence $(h_i)$
in $H_+$ such that the subsets $Bh_i$ are pairwise disjoint. Fix
$g\in V$. We have

$$\|u-\rho(g) u\|_X=\|\rho(h_i)u-\rho(g)\rho(h_i)
u\|_X\quad\text{(by
$H$-invariance of $u$)}$$
$$=\|\rho(h_i)u-\rho(h_i)\rho(g) u\|_X\quad\text{(as $u$ and $h_i$
commute)}$$
$$=\Delta(h_i)^{-1/p}\|u-\rho(g) u\|_{Xh_i},$$
which tends to 0 as $i\to\infty$ as $u-\rho(g) u$ is $L^p$.
Therefore, $\|u-\rho(g) u\|_X=0$ for every compact subset
$X\subset G$, so that $\|u-\rho(g) u\|=0$, i.e.~$u=\rho(g) u$
almost everywhere. \end{proof}

Fix $\xi\in G$ and $u\in D^p(G)$ satisfying
$\rho(\xi)u=u$. For $g\in G$ write $b(g)=u-\rho(g)u$.

\begin{lem}If $g\in G$ then $$\Delta(\xi)^{1/p}\|b(g)\|=\|b(\xi^{-1}g\xi)\|.$$\label{lem:conj}
\end{lem}
\begin{proof} This follows from the formula
$$u-\rho(g)u=\rho(\xi)(u-\rho(\xi^{-1}g\xi)u).\qedhere$$
\end{proof}

\begin{lem} Suppose that $\Delta(\xi)\ge 1$. If $u\in D^p(G)$ and $\rho(\xi)u=u$, then $u$ is invariant by $\rho(W_\xi)$.\label{lem:contract}
\end{lem}
\begin{proof} Using that $b$ is bounded
on bounded subsets of $G$, this follows
from Lemma \ref{lem:conj}. \end{proof}

\begin{lem}
Let $G$ be a locally compact group.\begin{itemize}

\item If $f\in L^p(G)$ ($p<\infty$), then $f$ cannot be left
or right invariant under a non-compact closed subgroup $H$ unless $f=0$.
\item If $f\in D^p(G)$ is invariant under a non-compact closed normal subgroup $N$, then $f$ is constant.\end{itemize}\label{lem:inv}
\end{lem}
\begin{proof}~
\begin{itemize}\item Otherwise, there exists $\varepsilon>0$ such that $W=\{|f|>\varepsilon\}$ has nonzero finite measure $m$.
Take a compact subset $K$ such that the measure of $K\cap W$ has
measure $>m/2$. There exists $g$ such that $K$ and $gK$ (resp.
$Kg$)are disjoint; we can suppose that $\Delta(g)\ge 1$. Then as
$gK$ (resp. $Kg$) are contained in $W$, we get a contradiction.
\item Then $b(g)=f-\rho(g)f$ is (left-)invariant by $N$, so is
zero, i.e.~$f$ is constant.\qedhere
\end{itemize}
\end{proof}

\subsection{A criterion for vanishing of $L^p$-cohomology}

\begin{prop}
Suppose that $G$ contains
\begin{itemize}
\item An element $\xi$ satisfying $\Delta(\xi)>1$;
\item Two non-compact closed subgroups $Z$ and $Y$, with $Z$ normal in $G$
\end{itemize}
and assume that $Y\subset W_\xi$ and $Y$ centralizes $Z$. Then
$H_p^1(G)=0$ for all $p\ge 1$.\label{lem:nonHeintze}
\end{prop}

\begin{proof} Take $u\in D^p(G)$ and let us show that it is in the cohomology class of
0. By Lemma \ref{lem_xinv} we can suppose that $\rho(\xi)u=u$. By
Lemma \ref{lem:contract}, $u$ is then invariant by $\rho(Y)$. By
Lemma \ref{lem:centre}, $u$ is invariant by $Z$, and therefore by
Lemma \ref{lem:inv} $u$ is constant.\end{proof}

\begin{rem}
In general, it is not true that if $H$ is a closed non-compact subgroup contained in $\textnormal{Ker}\Delta$ and if $u\in D^p(G)$ is $\rho(H)$-invariant, then $u$ is necessarily zero in $H_p^1(G)=0$. It is shown in \cite[Proposition 4.3]{CTV} that if $H$ is any infinite discrete group and $K$ any non-trivial discrete group, then the free product $G=H\ast K$ is a counterexample.
\end{rem}

\subsection{Application to non-unimodular amenable Lie or $p$-adic groups}

We say that a connected Lie group, resp.~connected algebraic group over a local field is triangulable if it embeds as a closed subgroup of upper triangular real matrices (resp.~ over the ground field).

\begin{prop}\label{MainLem}
Let $G$ be
\begin{itemize}
\item either a non-unimodular triangulable Lie group,
\item or a non-unimodular amenable connected linear algebraic group over a non-Archimedean local field of characteristic 0.
\end{itemize}
Then $G$ satisfies the hypotheses of Proposition \ref{lem:nonHeintze} if
(and only if) $G$ is not of Heintze type (Lie or non-Archimedean).
\end{prop}
\begin{proof}
The ``only if" part is not the point of this section; it follows from Propositions \ref{lem:nonHeintze} and \ref{ct}. Let us focus on the ``if" part, assuming that $G$ is not of Heintze type.

If $G$ is non-Archimedean, removing a maximal anisotropic torus in a Levi decomposition, we see that it has a cocompact Zariski-closed normal subgroup which is triangulable (and necessarily non-Heintze as well), so we can also suppose in that case that $G$ is triangulable (the group $Z$ will be chosen characteristic so will remain normal in the whole group).

Let $N$ be the nilpotent radical and $Z$ its center; the assumptions imply that $Z$ is not compact. 
Write a Cartan decomposition $G=AN$; this means that $A$ is nilpotent and $G=AN$ (in the non-Archimedean case we can have moreover $G=A\ltimes N$ and $A$ abelian). 

Consider the adjoint action of $A$ on $\mathfrak{g}$. It defines a homomorphism $A\to\textnormal{GL}(\mathfrak{g})$. Its Zariski closure is connected nilpotent, so decomposes as a direct product $DU$ with $D$ diagonalizable (over $\mathbf{K}$) and $U$ unipotent. If $a\in A$, we can thus decompose the corresponding automorphism of $\mathfrak{g}$ as $du$ and write (after choice of a suitable basis) $d=\textnormal{diag}(\lambda_1(a),\dots,\lambda_m(a))$. 
Define the {\it weights} as the homomorphisms $\omega_i:A\to\mathbf{R}$ defined by $\omega_i(a)=\log(|\lambda_i(a)|)\in\mathbf{R}$. Set $B=\cap_i\textnormal{Ker}(\omega_i)$, so $A/B$ is isomorphic to $\mathbf{R}^k$ (Lie case) or $\mathbf{Z}^k$ (non-Archimedean case) for some $k$. For every $a\in A$, we have $\Delta(a)=\exp(\sum_i\omega_i(a))$. Set $$E_+=\{a\in A|\Delta(a)>1\}=\left\{a\in A|\sum_i\omega_i(a)\ge 0\right\},$$ which can be viewed as a ``half-space" in $A/B$. Set
$$E_{\text{dil}}=\{a\in A|\forall i,\omega_i(a)\ge 0\}.$$
Clearly, $E_{\text{dil}}\subset E_+$, but in $A/B$, $E_{\text{dil}}/B$ identifies with $\mathbf{R}_+^k$ or $\mathbf{N}^k$.

Therefore, if $k\ge 2$, then there exists an element $\xi$ in $E_+$ which is not in $E_{\text{dil}}$. So $\omega_i(\xi)<0$ for some $i$. This means that $W_\xi$ is not compact and the hypotheses are fulfilled.

Let us now suppose that $k=1$ ($k=0$ is ruled out as it would force $G$ to be unimodular (and even nilpotent)). Note that this forces $\mathfrak{n}$ to have codimension one. Pick $\xi$ with $\Delta(\xi)>1$. As the action of $\xi$ on $\mathfrak{n}$ has at least one eigenvalue of modulus $\le 1$ (because $G$ is not Heintze), we have $W_\xi$ non-compact.
\end{proof}

\begin{rem}
Let $\mathbf{R}^2$ act on $\mathbf{C}^2$ by $(r,\theta)\cdot(z_1,z_2)=r(e^{i\theta}z_1,e^{\pi i\theta}z_2)$. Then the semidirect product $\mathbf{R}^2\ltimes\mathbf{C}^2$ does not satisfy the hypotheses of Proposition \ref{lem:nonHeintze}, yet is non-Heintze. Note that this group is not triangulable.
\end{rem}

\begin{lem}
Let be a Lie group with $\pi_0(G)$ finite, and assume that $G$ is of Heintze type with $G=S\ltimes N$ as in Definition \ref{dco}. Set $K=\textnormal{Ker}(\Delta_G)\cap S$. Then $K$ is a compact normal subgroup of 
$S$ and has a direct factor in $S$, isomorphic to $\mathbf{R}$. Moreover, any element $g$ in $S-K$ generates a cocompact subgroup contracting $N$.\label{allco}
\end{lem}
\begin{proof}
As $S_0$ contains a cocompact cyclic subgroup, we have $K\cap S_0$ compact, hence $K$ is compact. Since both $S$ and $N$ are unimodular, $\Delta_G$ is non-trivial in restriction to $S$, hence to $S_0$. Note that $K$ is a maximal compact subgroup in $S$ and in particular $K\cap S_0=K_0$. As $K_0$ is a connected compact Lie group, its inner automorphism group is open. In particular, the action of $S_0$ by conjugation on $K_0$ is by inner automorphisms and therefore its kernel $S_1$ satisfies $S_0=K_0S_1$, in particular $S_1$ is cocompact in $S$. Let $L$ denote a one-parameter subgroup of $S_1$, not contained in $\textnormal{Ker}(\Delta_G)$. In particular, $L$ is closed and non-compact; by construction it centralizes $K$. As $\Delta_G$ is surjective in restriction to $L$, we have $S=KL$ and as $[K,L]=1$, $K$ is compact and $K\cap L=1$, this is a (topological) direct product.

Let $\xi$ be an element of $S$ contracting $N$. For some suitable norm, the adjoint action of $K$ on $\mathfrak{n}$ is isometric, and the action of the element $\xi$ is strictly contracting. For some $\gamma\in L$, we have $\gamma k=\xi$ for some $k\in K$. Hence $\gamma$ contracts $\mathfrak{n}$. So it has all its eigenvalues on $\mathfrak{n}$ of modulus $<1$. Therefore all $(\gamma^t)_{t>0}$ is contracting, and therefore any element of the form $\gamma^tk$ with $t>0$ and $k\in K$ is contracting; each cyclic subgroup of $S$ not contained in $K$ contains such an element.
\end{proof}

\begin{lem}\label{coch}
Let $G_1$ be a Lie group of Heintze type and $G_2$ a connected, cocompact normal subgroup. Then $G_2$ is of Heintze type.
\end{lem}
\begin{proof}
Write $G_1=S\ltimes N$ as in the definition of Heintze type and let $\xi$ be a contracting element. Then both $\mathfrak{s}$ and $\mathfrak{n}$ are sums of characteristic subspaces for the adjoint action of $\xi$. Therefore, $\mathfrak{g}_2=\mathfrak{s}_2\ltimes\mathfrak{n}_2$ for some subspaces (necessarily ideals) $\mathfrak{s}_2$ and $\mathfrak{n}_2$ of $\mathfrak{s}$ and $\mathfrak{n}$, hence $G_2=S_2\ltimes N_2$ with $S_2$ and $N_2$ normal subgroups of $S$ and $N$ respectively. As $G_2$ is cocompact and connected, necessarily $N_2$ is cocompact and connected, hence $N_2=N$. In view of Lemma \ref{allco}, $S_2$ contains a cocompact cyclic subgroup contracting $N$, so $G_2$ is of Heintze type.
\end{proof}

\begin{lem}Let $G$ be a Lie group with $\pi_0(G)$ finite. Let $G_1$ be a cocompact, normal, contractible subgroup of $G$. Then $G=K\ltimes G_1$ for some compact subgroup (actually, any maximal compact subgroup).\label{homo}
\end{lem}
\begin{proof}Let $K$ be a maximal compact subgroup of $G$. From the exact sequence associated to a fibration, we see that the natural map $G\to G/G_1$ is a homotopy equivalence. Also, the inclusion $K\subset G$ is a homotopy equivalence. So $K\to G/G_1$ is a homomorphism between compact Lie groups which is a homotopy equivalence. Again using the long exact sequence associated to a fibration, we obtain that both the kernel and the cokernel of $K\to G/G_1$ are contractible. As these are compact manifolds, they are necessarily points, that is, the map $K\to G/G_1$ is an isomorphism. So $G=K\ltimes G_1$.
\end{proof}

\begin{lem}\label{ncar}
Let $G$ be a Lie group with $\pi_0(G)$ finite. Assume that $G$ is of Heintze type with $G=S\ltimes N$ as in Definition \ref{dco}. Then $G/N$ is the largest quotient of $G$ with polynomial growth (also known as ``exponential radical of $G$"). In particular, $N$ is a characteristic subgroup of $G$.
\end{lem}
\begin{proof}
Obviously $G/N$ has polynomial growth. Conversely, let $M$ be a normal, closed subgroup of $G$ such that $G/M$ has polynomial growth and let us show that $N\subset M$. By Guivarc'h \cite{Gui}, this means that in the adjoint representation of $G$ on $\mathfrak{g}/\mathfrak{m}$, all eigenvalues have modulus one. This forces $\mathfrak{n}\subset\mathfrak{m}$, so $N\subset M$.
\end{proof}

\begin{lem}\label{hcoc}
Let $G$ be a Lie group with $\pi_0(G)$ finite. Assume that $G$ has a normal, cocompact subgroup $G_1$ which is of Heintze type. Then $G$ is of Heintze type.
\end{lem}
\begin{proof}
Write $G_1=S\ltimes N$ as in Definition \ref{dco}. By Lemma \ref{ncar}, $N$ is characteristic in $G_1$ and therefore is normal in $G$.

By Lemma \ref{allco}, $S=K_1\times T_1$ with $T\simeq\mathbf{R}$. Let us prove that $T\ltimes N$ is normal in $G$. The subgroup $T\ltimes N$ is contained in the radical $R_1$ of $G_1$. Now $R_1$ is characteristic in $G_1$, so is normal in $G$, so is contained in the radical $R$ of $G$. Hence $T\ltimes N\subset R$. Therefore the image of $T$ in $G/N$, which we still denote by $T$, is contained in the radical $R/N$ of $G/N$. Since $G_1$ is normal and cocompact, the restriction of $\Delta_G$ to $G_1$ is $\Delta_{G_1}$. So the map $\Delta_G$ can be viewed as a homomorphism on $G/N$, which is non-trivial on $T$, hence on $R$. Set $M=R\cap\textnormal{Ker}(\Delta)$, so $R/M\simeq\mathbf{R}$ and $M/N$ is compact. Since $R/N$ is connected and $M/N$ is its maximal compact subgroup, $M/N$ is connected. As a solvable, connected compact normal subgroup of $R/N$, $M/N$ is a central torus. Since it has codimension one in $R/N$, there is one-dimensional factor, hence a direct factor of $M/N$ in $R/N$. So $R/N$ is abelian. In particular the action of $G$ on $R/N$ by conjugation factors through the compact group $G/R$. This group preserves a direct product decomposition $R/N=M/N\oplus V$; since $G/R$ is compact, it action on the compact torus $M/N$ is trivial. Moreover, its action preserves leaves invariant the function $\Delta_G$ and therefore acts trivially on $V$. So $R/N$ is central in $G/N$,
hence $T$ is normal in $G/N$. So $T\ltimes N$ is normal in $G$.
\end{proof}

\begin{lem}\label{qtc}
Let $G$ be a connected amenable Lie group and $W$ a compact normal subgroup. If $G/W$ is of Heintze type, then so is $G$.
\end{lem}

\begin{proof}
First observe that the result is clear when $W$ is locally a direct factor. In general, we can suppose that $W$ is either semisimple, finite, or a circle. When $W$ is semisimple, as its outer automorphism group is discrete and $W$ has finite center, its centralizer is locally a direct factor. 

Suppose that $W$ is finite. In particular, $W$ centralizes $G_0$. Using Lemma \ref{allco}, write $G/W=(T/W\times K/W)\ltimes N/W$. As $T/W$ and $N/W$ are both simply connected, we have $T=T_0\times W$ and $N=N_0\times W$. So $G=TK\ltimes N$ is of Heintze type.

Finally, suppose that $W$ is a circle and that $G/W$ is of Heintze type. By Lemmas \ref{coch} and \ref{hcoc}, we can suppose that $G/W=S/W\ltimes N/W$ with $S/W\simeq\mathbf{R}$. As $S/W$ is one-dimensional, we can lift it to a one-parameter subgroup $T$ of $G$, so $G=T\ltimes N$ and $N\supset W$. 

Let $\xi\in T$ be an element contracting $\mathfrak{n}$. Consider the adjoint action of $\xi$ on $\mathfrak{n}$, and denote by $\mathfrak{c}$ be the sum of characteristic subspaces for eigenvalues of modulus $<1$. 
Since $\mathfrak{c}=\{x\in\mathfrak{n}|\lim_{n\to +\infty}\text{Ad}(\xi^n)x\to 0\}$, we see that $\mathfrak{c}$ is a Lie subalgebra.
Note that since $G/W$ is Heintze, $\mathfrak{c}$ projects onto $\mathfrak{n}/\mathfrak{w}$, that is, $\mathfrak{c}+\mathfrak{w}=\mathfrak{n}$. Clearly the intersection is trivial (since $\mathfrak{w}$ is central), so $\mathfrak{c}\oplus\mathfrak{w}=\mathfrak{n}$. Again using that $\mathfrak{w}$ is central, we obtain that this is a decomposition as a direct product of Lie subalgebras. Let $C$ be the Lie subgroup corresponding to $\mathfrak{c}$. The projection $C\to N/W$ inducing the isomorphism $\mathfrak{c}\to\mathfrak{n}/\mathfrak{w}$, it is a covering; as $G/W$ is simply connected, this is an isomorphism. In particular, $C$ is closed and nilpotent, simply connected. This implies in particular $C\cap W=1$. So $N$ is the topological direct product of $C$ and $W$. As $T$ normalizes both $C$ and $W$, we obtain that $W$ is a direct factor in $G$, hence clearly $G$ is of Heintze type.
\end{proof}

\begin{thm}\label{annu}
Let $G$ be a non-unimodular amenable Lie group with finitely many components, or a non-unimodular amenable linear algebraic group over a local field of characteristic zero. Assume that $G$ is not of Heintze type (Lie or non-Archimedean). Then $H^1_p(G)=0$ for all $p\ge 1$.
\end{thm}
\begin{proof}
If $G$ is triangulable or non-Archimedean, this is Proposition \ref{MainLem}. The remaining of the proof concerns connected Lie groups, and consists in reducing to the triangulable case, using the fact that the $L^p$-cohomology is a quasi-isometry invariant (see Appendix \ref{qii}).

Assume that $G$ is an amenable Lie group with $\pi_0(G)$ finite and $G$ not of Heintze type. Then the radical $G_1$ of $G$ is cocompact and by Lemma \ref{hcoc} is not of Heintze type. Let $W$ be the maximal normal compact subgroup in the connected solvable Lie group $G_1$. By Lemma \ref{qtc}, $G_2=G_1/W$ is not of Heintze type. Note that $G_2'$ is simply connected. By the {\it trigshadow} construction \cite[Lemma~2.4]{Cor}, there exist connected Lie groups with normal cocompact inclusions $G_2\subset G_3\supset G_4$ with $G_4$ triangulable. By Lemmas \ref{coch} and \ref{hcoc}, $G_4$ is not of Heintze type and we can conclude by Proposition \ref{MainLem}, using that cocompact inclusions are quasi-isometries.
\end{proof}

\subsection{The case of non-amenable Lie or $p$-adic groups}

\begin{lem}\label{nonaLIE}
Let $G$ be a connected non-amenable Lie group which is not of rank-one type. Let $W$ be the maximal compact normal subgroup in $G$. Then $G$ has a cocompact subgroup $L$, containing $W$, such that $L/W$ is a simply connected solvable Lie group which is not of Heintze type.  
\end{lem}
\begin{proof}Let $M$ be the connected amenable radical of $G$.

Assume that $M$ is not compact and $G/M$ has finite center.
Let $R$ be the radical of $G$ ($R$ is cocompact in $M$). Modding out if necessary, we can suppose that $W=1$. So the derived subgroup $R'$ of $R$ is simply connected. Consider the action by conjugation of $G/R$ on $R/R'$. As $G/R$ is semisimple, using the action on the universal covering, we obtain that $R/R'$ decomposes, under the action of $G$, as the direct sum of its maximal compact subgroup, and some vector space $V_0$. Let $V$ be the inverse image of $V_0$ in $R$. So $V$ is normal in $G$, simply connected, and cocompact in $M$. The group $G/V$ is the direct product, up to some normal finite subgroup, of a compact group and a semisimple group with finite center. Therefore it has a simply connected solvable cocompact subgroup; let $P$ be its inverse image in $G$, which is a simply connected solvable cocompact subgroup of $G$. Assume by contradiction that $P$ is Heintze. We know that the derived subgroup of $P/V$ has codimension at least one in $P/V$. As the Heintze assumption implies that $P'$ has codimension one, this forces $P'$ to contain $V$. As $P$ is Heintze, it contains an element $\xi$ contracting $P'$. In particular, $\xi$ contracts $V$. If $g\in G$, let the action of $g$ by conjugation on $V$ multiply the Haar measure by $q(g)$. So, using that $V$ is non-compact, $q(\xi)\neq 1$. Because of the existence of this contraction, $V\subset G'$ and in particular $V\subset\textnormal{Ker}(q)$, i.e.~$q$ factors through $G/V$. Now $G/V$ has compact abelianization because it is compact-by-semisimple. So $q$ is trivial, a contradiction.

Assume that either $M$ is compact or $G/M$ has infinite center. If $M$ is compact, then by assumption $G/M$ has rank at least two, so $G/M$ either has infinite center or has rank at least two. Then by Lemmas 2.4 and 6.7 in \cite{Cor}, there exist cocompact inclusions $$G/W\supset G_1\subset G_2$$ with $G_2$ a solvable and simply connected Lie group. The assumptions imply that $G/M$ contains $\mathbf{Z}^2$ as quasi-isometrically embedded subgroup this copy can be lifted to a Levi factor of $M$ in $G/W$, so $G/W$ also contains a quasi-isometrically embedded subgroup isomorphic to $\mathbf{Z}^2$. 
Therefore $G_2$ contains a quasi-isometrically embedded copy of $\mathbf{Z}^2$ as well (actually, it follows from the construction in \cite{Cor} that it can be realized as a subgroup, even if we can bypass it), so is not of Heintze type.
\end{proof}

If $G$ is a linear algebraic group over a local field $\mathbf{K}$, the rank (or $\mathbf{K}$-rank) of $G$ is the least $k$ such that $G$ contains a $\mathbf{K}$-split torus of rank $k$.

\begin{lem}\label{nau}
Let $G$ be a connected linear algebraic group over a local field $\mathbf{K}$ of characteristic zero. Suppose that $G$ is not amenable, and is not reductive of rank one.
Then $G$ contains a cocompact subgroup of the form $DU$ ($D$ split torus, $U$ unipotent) which is not of non-Archimedean Heintze type.
\end{lem}
\begin{proof}This is a simplified analog of the proof of Lemma \ref{nonaLIE}, so we only sketch. The hypotheses mean that either $G$ has rank at least two, or is not reductive.

Let $N$ be the unipotent radical of $G$. The reductive group $G/N$ has a cocompact subgroup of the form $DU'$ with $D$ a split torus and $U'$ unipotent. Let $P=DU$ be the inverse image of this subgroup in $G$. If $G$ has rank at least two, so does $P$, so that $P$ is not of non-Archimedean Heintze type. Otherwise, since we assume that $G$ is not amenable, $G/N$ is non-abelian reductive of rank one. In particular, $G/N$ has no non-trivial homomorphism to $\mathbf{R}$, as well as $N$. Therefore the action of $G$ on $N$ by conjugation preserves the Haar measure. As in the Proof of Lemma \ref{nonaLIE}, and as $N$ is non-trivial, this prevents $P$ from being a non-Archimedean Heintze group.
\end{proof}

\begin{prop}
Let $G$ be either a connected Lie group, or a linear algebraic group over a local field of characteristic zero. Suppose that $G$ is not amenable, and not of rank-one type. Then $H^1_p(G)=0$ for all $p\ge 1$. 
\end{prop}
\begin{proof}The hypotheses exactly mean that Lemmas \ref{nonaLIE} and \ref{nau} do apply. So $G$ is quasi-isometric to a non-Heintze simply connected solvable Lie group, or non-Heintze triangulable group over a local field of characteristic zero. Again using the quasi-isometry invariance of the $L^p$-cohomology (Appendix \ref{qii}), the result then follows from Theorem \ref{annu}.
\end{proof}

\section{The $L^p$-cohomology of Heintze groups}\label{MainThmSection}

In this section, we prove (2) and ($2'$) of Theorem \ref{main}.

We consider semidirect products $D\ltimes N$ with the following convention: $D$ acts on $N$ on the right, and the group law is
given by
$$(d_1,n_1)(d_2,n_2)=(d_1d_2,(n_1\cdot d_2)n_2).$$

Suppose that $D$ is discrete. Then it is checked at once that if
$\lambda_0$ is a left Haar measure on $N$, and $\delta$ is the
counting measure on $D$, then $\lambda=\delta\otimes\lambda_0$ is
a left Haar measure on $G=D\ltimes N$.

Suppose now that $u$ is a continuous function on $G$ that is invariant for
the right-regular action of $D$ (i.e.~$u(g\cdot d)=u(g)$ for all
$g\in G$, $d\in D$). Then, if we set, for $g\in N$, $v(g)=u(1,g)$,
we have $u(d,n)=v(n\cdot d^{-1})$ for all $(d,n)\in D\ltimes N$.

\begin{lem}If $D$ is infinite and $v$ is not constant, then $u$ is not in $\mathbf{R}+L^p(G)$.\end{lem}
\begin{proof}
For some $\varepsilon>0$, there exist two disjoint measurable subsets of positive measure $A_1,A_2$ of $N$ such that for every $(a_1,a_2)\in A_1\times A_2$, $v(a_1)\le v(a_2)-\varepsilon$. As $D$ is infinite, the subset $D_+$ of $D$ consisting of elements such that the automorphism $g\mapsto g\cdot d$ of $N$ dilates the measure, is infinite. If $d\in D_+$ and $a\in A_i$, then $u(d,a_i\cdot d)=v(a_i)$. Therefore for every $(b_1,b_2)$ with $$b_i\in B_i=\bigcup_{d\in D_+} \{d\}\times(A_i\cdot d) ,$$
we have $u(b_1)\le u(b_2)-\varepsilon$. As $B_1,B_2$ both have infinite measure, this implies that $u\notin \mathbf{R}+L^p(G)$.
\end{proof}

\begin{lem}
Set, for $g\in N$, $b(g)=v-\rho(g)v$. Then $u\in D^p(G)$ if and
only if $v\in D^p(N)$ and
\begin{equation}\sum_{d\in D}\Delta(d)\|b(g\cdot
d^{-1})\|^p_p\label{eq:b}\end{equation}
is finite for all $g\in N$ and tends to 0 when $g\to 1$.
\end{lem}
\begin{proof}
For $g\in G$, set $B(g)=u-\rho(g)u$. To check $B(g)\in L^p(G)$ for all $g\in G$, it is enough to check it for $g$ ranging over a generating subset of $G$. As $B(g)=0$ for $g\in D$, it is enough to check it for $g\in N$.

For fixed $d\in D$, we have
$v_d(n):=u(d,n)=v(n\cdot d^{-1})$, and the condition is that
$$\sum_{d\in D}\|v_d-\rho(g)v_d\|_p^p<\infty$$
for all $g\in N$.

Now $(v_d-\rho(g)v_d)(n)=b(gd^{-1})(nd^{-1})$, so
that $$\left(\textnormal{reminder: }\int
f(gh)dg=\Delta(h)^{-1}f(g)dg\right)$$
$$\|v_d-\rho(g)v_d\|^p=\Delta(d)\|b(g\cdot
d^{-1})\|^p_p.$$
Thus, for $g\in N$ we have $$\|B(g)\|_p^p=\sum_{d\in D}\Delta(d)\|b(g\cdot
d^{-1})\|^p_p\qedhere$$
\end{proof}

Let us now specify to the case when $D$ is cyclic and generated by
an element $\xi$ satisfying $\delta=\Delta(\xi)>1$.

Let $W$ be the largest compact normal subgroup in the unit component $H_0$ of $H$, look at the eigenvalues of $\xi$ on the Lie algebra on the non-trivial simply connected Lie group $N_0/W$, and define $\lambda$ as the minimal modulus of its (complex) eigenvalues.

\begin{thm}\label{mixx}
Consider a locally compact group $G=S\ltimes H$. Assume that some cyclic cocompact subgroup of $S$ contracts $H$, and that $H_0$ is non-compact. Define $\delta,\lambda$ as above, and $G=\mathbf{Z}\ltimes H$. Then $H^1_p(G)=0$ if and only if
$$p\le p_0(G):=\log(\delta)/\log(\lambda).$$
\end{thm}
\begin{proof}
Using the quasi-isometry invariance of $L^p$-cohomology (Appendix \ref{qii}), we can suppose that $S=\mathbf{Z}$ contracts $H$.

By Corollary \ref{confi}, $H/W(H_0)$ has a characteristic open subgroup $L$ decomposing canonically as a direct product of a connected Lie group and a totally disconnected characteristic subgroup. Therefore the theorem is a combination of the two following propositions.
\end{proof}

\begin{prop}\label{ct}
Consider contracting actions $\sigma$, $\sigma'$ of $\mathbf{Z}$ on a connected Lie group $N$ and on a totally disconnected locally compact group $R$, and $H=N\times R$ with the diagonal (contracting) action of $\mathbf{Z}$. Define $G=\mathbf{Z}\ltimes H$. Pick $p\ge 1$. If $N$ is non-compact, assume in addition that $p>p_0(G)$. Then $H^1_p(G)\neq 0$.
\end{prop}
\begin{proof}
On $H=N\times R$, we define $v(n,r)=w(n)\beta(r)$, where $w$ is non-zero Lipschitz and compactly supported on $N$ (Lipschitz referring to the intrinsic Riemannian distance on $N$), and $\beta$ being the indicator function of some clopen neighborhood of $1$.

As $b$ is bounded, for $d\to -\infty$, the sum (\ref{eq:b}) converges in $L^p$-norm, uniformly on $g$.

Then for fixed $g=(h,s)\in N=N\times R$,
$$b(g\cdot \xi^{-d})(n,r))=v(n,r)-v((n,r)(g^{-1}\cdot \xi^{-d}))$$ 
$$=w(n)\beta(r)-w((n(h^{-1}\cdot \xi^{-d}))\beta(r(s^{-1}\cdot \xi^{-d})))$$
When $d\to\infty$, $s^{-1}\cdot \xi^{-d}\to 1$, so eventually, say for $d\ge d_0=d_0(g)$, we have $\beta(r\cdot(g^{-1}\cdot \xi^{-d}))=\beta(r)$. So for $d\ge d_0$
$$b(g\cdot \xi^{-d})(n,r))=\beta(r)[w(n)-w(n(h^{-1}\cdot \xi^{-d}))]$$
Pick $\lambda'<\lambda$ such that $\delta<\lambda'^p$. Then as $\lambda'<\lambda$, the Riemannian length $h^{-1}\cdot \xi^{-d}$ is bounded above by $\lambda'^{-d}$ for $d\gg 0$. As $w$ is Lipschitz, this implies that for some constant $C$,
$$b(g\cdot \xi^{-d})(n,r))\le C\lambda'^{-d}.$$
As $b(g')$ has support of bounded measure (independently on $g'$), this implies $\|b(g\cdot \xi^{-d})\|\le C'\lambda'^{-d}$ for $d\ge 0$, for some suitable constant $C'$. So 
$$\Delta(\xi)^d\|b(g\cdot \xi^{-d})\|^p\le C'(\delta/\lambda'^{-p})^d.$$
and (\ref{eq:b}) holds.

We still need the continuity at $1$. First note that $d_0=d_0(g)$, as defined above, can be chosen bounded when $g$ is bounded. We pick $d_0$ working for some neighborhood $V$ of 1. From the Lipschitz condition, we actually have

$$|b(g\cdot \xi^{-d})(n,r))|\le C|h^{-1}\cdot \xi^{-d}|,$$
where $|\cdot|$ denotes the Riemannian length in $N_0$, and
$|h^{-1}\cdot \xi^{-d}|\le C'|h|\lambda'^{-d}$ for all $d\ge d_1$, $d_1$ being independent on $h\in V$, which lies in the given neighborhood of $1$. This implies that the sum
$$\sum_{d\ge d_1}\Delta(\xi)^d\|b(g\cdot \xi^{-d})\|^p$$
is continuous at 1, and therefore so is the entire sum, indexed by $\mathbf{Z}$.
\end{proof}

\begin{prop}
Consider contracting actions $\sigma$, $\sigma'$ of $\mathbf{Z}$ on a non-compact connected Lie group $N$ and on a totally disconnected group $R$, and $H=N\times R$ with the diagonal (contracting) action of $\mathbf{Z}$. Define $G=\mathbf{Z}\ltimes H$ and $p_0(G)$ as above.
Then $H^1_p(G)=0$ for all $p\le p_0(G)$.
\end{prop}
\begin{proof}
The first step is to modify the action of $\mathbf{Z}$ on $N$ so as to have a triangulable action with positive real eigenvalues.

Consider the tangent action of $\sigma(\mathbf{Z})$ on $\mathfrak{n}$. Then we have $\mathfrak{n}=\mathfrak{k}\oplus\mathfrak{m}$, where $\mathfrak{m}$, resp.~$\mathfrak{k}$ is the sum of characteristic subspaces for eigenvalues of $\sigma(\mathbf{Z})$ of modulus less than 1 (resp.~equal to 1). Moreover these are Lie subalgebras, and $[\mathfrak{k},\mathfrak{m}]\subset\mathfrak{m}$, so $\mathfrak{m}$ is an ideal. Let $M$ and $K$ be the corresponding Lie subgroups of $G$. Then $M$ is strictly contracted by $\sigma(\mathbf{Z})$, so is nilpotent and simply connected. Since $M$ is contained in the nilpotent radical of $N$, which is simply connected, $M$ is closed. Moreover $\sigma$ contracts the quotient $N/M$, so that the tangent action has all its eigenvalues of modulus one. This implies (for instance) that $\sigma(\mathbf{Z})$ preserves the measure on $N/M$, and therefore $N/M$ is compact. So $\mathbf{Z}\ltimes M\times R$ is cocompact in $G$, hence quasi-isometric to it. In view of the quasi-isometry invariance of the $L^p$-cohomology (Appendix \ref{qii}), we can replace $N$ by $M$ if necessary, so we can suppose that $N$ is nilpotent and simply connected.

Then the group $\textnormal{Aut}(N)=\textnormal{Aut}(\mathfrak{n})$ is a linear algebraic group; so we can write $\sigma(1)=g_+k=kg_+$ with $k$ elliptic and $g_+$ having all eigenvalues real positive. If we define a new action $\sigma_1$ of $\mathbf{Z}$ on $N$ by replacing $\sigma(n)$ by $g_+^n$, define $G_1=\mathbf{Z}\ltimes_{\sigma_1,\sigma'}(N\times R)$. Let $K$ denote the closure the subgroup $\langle k\rangle$ of $\textnormal{Aut}(N)$. We can make $\mathbf{Z}\times K$ act on $N\times R$, the action on $N$ being the original action, the action of $\mathbf{Z}$ on $R$ being the original action, and the action of $K$ on $R$ being trivial. Then both $G$ and $G_1$ embed into $(\mathbf{Z}\times K)\ltimes (N\times R)$ as cocompact subgroups. Therefore $G_1$ is quasi-isometric to $G$. Again using the quasi-isometry invariance of $L^p$-cohomology, we can henceforth assume that $N$ is simply connected nilpotent and $\mathbf{Z}$ acts on it with all eigenvalues real positive.

Assume that $H^1_p(G)\neq 0$. By Lemma \ref{lem_xinv}, there exist $u\in D^p(G)$ which is $\rho(\xi)$-invariant, and therefore $u$ can be written as above. As the eigenvalues of $\xi$ on $N$ are real, there exists a one-parameter subgroup $\gamma(t)_{t\in\mathbf{R}}$ of $N$ on which $\xi$ acts by multiplication by $\lambda$.
Then (\ref{eq:b}) reads as
\begin{equation}\label{ee}\sum_{n\in\mathbf{Z}}\Delta(\xi)^n\|b(\gamma(\lambda^{-n}))\|^p_p<\infty.\end{equation}

Set $g_0=\gamma(1)$. For $n\ge 0$, write $\lambda^n=m_n+\varepsilon_n$ where
$m_n=\lfloor\lambda^n\rfloor$.
Then$$\|b(g_0)\|=\|b(\gamma(1))\|=\|b(\gamma((m_n+\varepsilon_n)\lambda^{-n}))\|$$
$$\le
m_n\|b(u(\lambda^{-n}))\|+\|b(\varepsilon_n\lambda^{-n})\|.$$
In the last inequality, we use that $\rho|_N$ is an isometric action.  Set
$$e_n=\sup\{\|b(\gamma(t))\|:t\in[0,\lambda^{-n}]\}.$$ Then 
$$\|b(g_0)\|\le
\lambda^n\|b(\gamma(\lambda^{-n}))\|+e_n.$$
We have $e_n\to 0$ for $n\to+\infty$ since $\lambda>1$ and the cocycle $b$ is continuous.
Assume that $\|b(g_0)\|\neq 0$. Then for $n\gg 0$ we have
$$0\le\|b(g_0)\|-e_n\le\lambda^n\|b(\gamma(\lambda^{-n}))\|,$$
so (\ref{ee}) implies 
$$\sum_{n\in\mathbf{N}}\left(\frac{\Delta(\xi)}{\lambda^p}\right)^n(\|b(g_0)\|-e_n)^p<\infty.$$
As $\|b(g_0)\|-e_n\to\|b(g_0)\|>0$, this implies $|\Delta(\xi)/\lambda^p|<1$, that is,
$$p>\log(\Delta(\xi))/\log\lambda.$$
Now assume by contradiction that $b(g_0)=0$, that is, $\rho(g_0)u=u$. Let $Z$ be the center of $N_0$, which is non-compact. By  
Lemma \ref{lem:centre}, $\rho(Z)u=u$. By Lemma \ref{lem:inv}, $u$ is constant on $G$, a contradiction.
\end{proof}

\section{On Gromov-hyperbolic groups}\label{hyp}

In the section, we prove the corollaries of Theorem \ref{main}. First, Corollary \ref{c2} is immediate from the theorem.

The following proposition is essentially proved, in the even more general context of metric groups by Gromov \cite[8.2.D]{G}.

\begin{prop}
Let $G$ be a non-elementary Gromov-hyperbolic locally compact group. Then $G$ contains a quasi-isometrically embedded free subsemigroup on two generators.\label{qf}
\end{prop}
\begin{proof}
Gromov provides two hyperbolic elements $\gamma_1,\gamma_2$ with origin $o_1,o_2$, target $t_1,t_2$ and $\#\{o_i,t_1,t_2\}=3$ for $i=1,2$. It then follows from the quasi-isometric ping-pong Lemma \cite[Lemma~2.1]{CT} that suitable powers of $\gamma_1$ and $\gamma_2$ generate a free subsemigroup.
\end{proof}

Such a quasi-isometri\-cally embedded subsemigroup provides a quasi-isometri\-cally embedded tree in $G$, and by \cite[Theorem~6]{T1}, this implies that the $L^p$-cohomology is non-zero for large $p$:

\begin{cor}
If $G$ is a non-elementary Gromov-hyperbolic locally compact group, then $H^1_p(G)\neq 0$ for $p$ large enough.
\end{cor}

Note that for the groups considered in the paper, the results of \cite{CT} are enough to provide a quasi-isometrically embedded subsemigroup without using Proposition \ref{qf}, but we found it natural and useful to mention it. Anyway, (\ref{c3}.\ref{3a})$\Rightarrow$(\ref{c2}.\ref{2a}) and thus (\ref{c3}.\ref{3a})$\Rightarrow$(\ref{c3}.\ref{3b}) is proved. The converse (\ref{c3}.\ref{3b})$\Rightarrow$(\ref{c3}.\ref{3a}) is a particular case of Corollaries \ref{c4} and \ref{c5} (although the reader can prove it directly in a more straightforward way).

In Corollary \ref{c4}, the implications 
(\ref{c4}.\ref{4c})$\Rightarrow$(\ref{c4}.\ref{4b})$\Rightarrow$(\ref{c4}.\ref{4a}) are straightforward. If (\ref{c4}.\ref{4a}) holds, by the already proved implication of Corollary \ref{c3}, the unit component $G_0$, and therefore $G$, is a Lie group of Heintze type or rank-one type.
So we have to prove the following proposition.

\begin{prop}
Let $G$ be a Lie group with finitely many components, of Heintze type or rank-one type, and $K$ a maximal compact subgroup. Then the connected manifold $G/K$ has a $G$-invariant Riemannian metric of negative curvature. 
\end{prop}
\begin{proof}
If $G$ is of rank-one type, $K$ acts irreducibly on the tangent space of the base-point of $G/K$, so the left-invariant Riemannian metric is unique up to scalar multiplication, so we get one of the simply connected irreducible symmetric spaces of rank one, which are negatively curved.

So assume that $G$ is of Heintze type.
Let $H$ be a cocompact normal subgroup which is simply connected.
Let $K$ be a maximal compact subgroup of $G$. The adjoint action of $K$ on $\mathfrak{h}$ preserves the hyperplane $\mathfrak{h}'$, so preserves a complement line $\mathfrak{a}$, corresponding to some one-parameter $A$ normalized by $K$. As the modular function $\Delta$ is non-trivial on $A$, an element of $A$ cannot be conjugate to its inverse, and therefore $K$ centralizes $A$.

Let $b$ be a $K$-invariant scalar product on $\mathfrak{h}$. Consider, for $\lambda>0$, the linear automorphism of $\mathfrak{h}$ mapping $a+v$ ($a\in\mathfrak{a}$, $v\in\mathfrak{h}'$) to $\lambda^{-1}a+v$. Note that it commutes with the action of $K$, so $(u_\lambda)_*b$ is also $K$-invariant.
Then \cite[Theorem 2]{H} states that if $\lambda$ is large enough, then the left-invariant metric on $H$ obtained by translating $(u_\lambda)_*b$ from the identity, is negatively curved. Moreover, this metric is $K$-invariant, and the group of isometries generated by left translations of $H$ and conjugation by elements of $K$ is naturally identified with $G=K\ltimes H$.
\end{proof}

Finally similarly, in Corollary \ref{c5}, the implications 
(\ref{c5}.\ref{5c})$\Rightarrow$(\ref{c5}.\ref{5b})$\Rightarrow$(\ref{c5}.\ref{5a}) are straightforward. If (\ref{c5}.\ref{5a}) holds, by the already proved implication of Corollary \ref{c3}, $G$ is of non-Archimedean Heintze or rank-one type. In the latter case, $G$ acts on the corresponding Bruhat-Tits tree, giving (\ref{c5}.\ref{5c}). Precisely, the action by conjugation of $G$ on $G_0$ provides a proper map $G\to\textnormal{Aut}(G_0)$, and $\textnormal{Aut}(G_0)$ has a natural action on the Bruhat-Tits tree. For groups of non-Archimedean type, we have Proposition \ref{nonArchimedeanHeintze} below. We first need the following lemma.



\begin{lem}\label{ac}
Let $H$ be a non-compact locally compact group with $H_0$ compact, endowed with a contracting action of $\mathbf{Z}$. Then there exists a vacuum subset which is a compact open subgroup.
\end{lem}
\begin{proof}
Let $U$ be a compact vacuum subset and set $L=\bigcap_{k\ge 0}k\cdot U$. Then it is easy to check that $L$ is a compact subgroup. So, as $H_0$ is compact, $L$ is contained in a compact open subgroup $V$ (by an easy argument; see if necessary \cite[Theorem~2]{Pey}), which is necessarily a vacuum subset.
\end{proof}

\begin{lem}\label{algc}
Let $H$ be a non-compact locally compact group with $H_0$ compact. Let $S$ be a locally compact group with an action $\pi$ on $H$ by group automorphisms, and suppose that $S$ possesses a cocompact copy of $\mathbf{Z}$, which contracts $H$. Then
\begin{itemize}
\item $S$ has a unique homomorphism $p$ onto $\mathbf{Z}$, which is positive on the given copy of $\mathbf{Z}$, and $W=\textnormal{Ker}(p)$ is compact.
\item there exists a compact open subgroup $\Omega$ of $H$ which is $\pi(W)$-invariant, and stable under $\pi(p^{-1}(\mathbf{N}))$, and for every compact subset $K$ of $H$ there exists $k$ such that $\pi(p^{-1}(\mathbf{N}_{\ge k}))K\subset \Omega$.
\end{itemize}
\end{lem}
\begin{proof}
Let us first check that $p$ is unique. Let $p$ be a surjective homomorphism $S\to\mathbf{Z}$. As $S$ contains a cocompact copy of $\mathbf{Z}$, $W=\textnormal{Ker}(p)$ has to be compact and is thus determined as the unique maximal normal compact subgroup of $S$. This gives only two possibilities for $p$, and uniqueness follows from the the positivity assumption.

If $g\in S$, let the automorphism $\pi(s)$ of $H$ multiply the Haar measure of $H$ by $q(s)\in\mathbf{R}_+^*$. As $H$ is non-compact, this is a non-trivial homomorphism $S\to\mathbf{R}^*_+$.
By Lemma \ref{ac}, we can choose a vacuum subset $L$ which is a compact open subgroup of $H$. Then $W_+=\{g\in S|\pi(g)L\subset L\}$ contains 1 in its interior. Indeed, otherwise there exists a net $g_i\to 1$ in $S$, $g_i\notin W_+$. So there exists $h_i\in L$ with $\pi(g_i)h_i\notin L$. As $L$ is compact, we can suppose that $(h_i)$ has a limit $h\in L$. As the action of $G$ on $H$ is continuous, $\pi(g_i)h_i\notin L$ converges to $h\in L$, a contradiction. Similarly $W_-=\{g\in S|\pi(g^{-1})L\subset L\}$ contains 1 in its interior, and therefore $W_0=W_+\cap W_-=\{g\in S|\pi(g)L=L\}$ is an open subgroup of $S$. Clearly, $W_0\subset W=\textnormal{Ker}(q)$. Therefore $W$ is open in $S$.
The group $S/W$ is discrete, is embedded into $\mathbf{R}$ (by $\log\circ q$), and contains a cocompact copy of $\mathbf{Z}$, so it is infinite cyclic as well. This yields the desired homomorphism $p$.

Let us now turn to the second assertion. It follows from the algebraic contractibility assumption that every compact subset of $H$ generates a relatively compact subgroup of $G$. Therefore the subgroup $L_1$ defined as the closed subgroup generated by the compact subset $\{\pi(g)h|g\in W,h\in L\}$ of $H$, is compact. Now $L_1$ is $\pi(W)$-invariant.

Fix $t\in p^{-1}(\{1\})$. Let $t'$ be the positive generator of the given copy of $\mathbf{Z}$. As $\pi(t)L_1$ is compact, there exists $k\ge 0$ such that $\pi(t'^k)\pi(t)L_1\subset L_1$. We can write $t'^kt=t^\ell w$ with $\ell\ge 1$ and $w\in W$. So $\pi(t^\ell)L_1\subset L_1$. Define $\Omega=\bigcap_{i=0}^{\ell-1}\pi(t^i)L_1$. Then $\Omega$ is compact, open, and $\pi(t)\Omega\subset\Omega$. Moreover, as $W$ is normalized by $S$, $\pi(t^i)L_1$ is $\pi(W)$-invariant for all $i$, so $\Omega$ is $\pi(W)$-invariant as well. Now $t$ and $W$ generate the semigroup $p^{-1}(\mathbf{N})$, so $\pi(p^{-1}(\mathbf{N}))\Omega\subset \Omega$. 

As $\bigcup_{i=0}^{\ell-1}\pi(t^i)L_1$ is a compact subset of $H$, there exists $\kappa$ such that
$$\pi(t^{\kappa})\left(\bigcup_{i=1}^{\ell}\pi(t^i)L_1\right)\subset L\subset L_1.$$ So $$\pi(t^{\kappa+i})L_1\subset L_1$$ for all $i=1\dots\ell$, that is,
$$\pi(t^{\kappa})L_1\subset\pi(t^{-i})L_1$$
for all $i=1\dots\ell$, that is
$$\pi(t^{\kappa+\ell})L_1\subset\bigcap_{j=0}^{\ell-1}\pi(t^j)L_1=\Omega.$$
Now if $K$ is a compact subset of $H$, there exists $k\ge 0$ such that $\pi(t^{k})K\subset L\subset L_1$. Therefore $\pi(t^{k+\kappa+\ell})K\subset \Omega$, so $\pi(g)K\subset \Omega$ whenever $p(g)\ge k+\kappa+\ell$.
\end{proof}

\begin{prop}\label{nonArchimedeanHeintze}
Keep the assumptions as in Lemma \ref{algc}. Then there exists a proper length function $\ell$ on the semidirect product $G=S\ltimes H$ such that the pseudo-metric space $(G,d)$, where $d(g,g')=\ell(g^{-1}g')$, is isometric to the vertex set of a $r$-regular tree for some $r\geq 3$.
\end{prop}
\begin{proof}
Let $W,p,H_0$ be as given by Lemma \ref{algc}, and fix $t\in p^{-1}(\{1\})$. Replacing if necessary $S$ by $\mathbf{Z}$ (but we still write it multiplicatively) and $H$ by $W\ltimes H$ (and $H_0$ by $W\ltimes H_0$), we can suppose that $G=\mathbf{Z}\ltimes H$, so now $p$ is just the identity.

Set $M=H_0\{t,t^{-1}\}H_0$. Consider the Cayley graph of $G$ with respect to $M$. As $M$ is invariant under conjugation by $H_0$, the right action of $H_0$ on $G$ preserves this Cayley graph structure. So we get a graph structure on the quotient $G/H_0$; moreover as the original right action of $H_0$ commutes with the left action of $G$, we get a left action of $G$ on the graph $G/H_0$. As $M$ generates $G$, this graph is connected. To check it is a tree, it is enough to check that $-p$ behaves like a Busemann function, i.e.~for any vertex $v$ with $p(v)=n$ there is only one vertex $v'$ adjacent to $v$ with $p(v')=n+1$. (Indeed, if we have an injective loop, it contains a vertex $v$ with $p(v)$ minimal, and the two adjacent vertices in the loop have $p(v')=p(v)+1$.) By homogeneousness, it is enough to check it when $v=(1,1)$ is the identity. Let $(t,u_1)$ and $(t,u_2)$ be two neighbours of $v$, viewed in $G$. This means that $(t,u_i)$ belong to $H_0tH_0$ for $i=1,2$. So there exist $v_1,v'_1\in H_0$ such that 
$$(t,u_1)=(1,v_1)(t,1)(1,v_2)=(1,v_1)(t,v_2)=(t,\pi(t)v_1.v_2)$$
As $\pi(t)v_1\in H_0$, we obtain that $u_1$ belongs to $H_0$. Similarly $u_2\in H_0$. Now
$$(t,u_1)^{-1}(t,u_2)=(1,u_1^{-1}u_2)\in H_0,$$
so $(t,u_1)$ and $(t,u_2)$ are identified in $G/H_0$.
\end{proof}

\appendix 

\section{Direct decomposition for a contraction}\label{a1}

This appendix is needed for the proof of Theorem \ref{mixx} (and thus Theorem \ref{mix}). It can also be of independent interest for the general study of contractions, as it generalizes \cite[Proposition 4.2]{Sie}, which applies to strict contractions.

We refer to Definition \ref{dco} for the definition of a contraction and a vacuum subset.
If a locally compact group $G$ has a maximal normal compact subgroup, such a subgroup is unique and denoted by $W(G)$. Notably, such a subgroup exists when $\pi_0(G)=G/G_0$ is compact and in particular when $G$ is connected. We say that a locally compact group $G$ is {\it elliptic} if every compact subset is contained in a compact subgroup.
We have the following easy lemma.

\begin{lem}\label{cell}
If a locally compact group $G$ has a contraction, then $G/G_0$ is elliptic.
\end{lem}
\begin{proof}
As $\alpha$ induces a contraction of $G/G_0$, we can suppose that $G$ is totally disconnected and we have to prove that $G$ is elliptic.
Let $U$ be a compact vacuum subset. Define the limit set $L$ as $\bigcap_{n\ge 0}\alpha^n(U)$.
This is a compact subgroup. Note that $\alpha$ (strictly) contracts $G$ modulo $L$, in the sense of \cite{HS}. Since $G$ is totally disconnected, there exists a compact open subgroup $V$ containing $L$. If $K$ is any compact subset of $G$, then $\alpha^n(K)$ is contained in $V$ for some $n$, and therefore $K$ is contained in the compact subgroup $\alpha^{-n}(V)$. 
\end{proof}

\begin{lem}\label{cod}
Let $G$ be a locally compact group with $\pi_0(G)$ compact and $W(G_0)=1$. Then the subgroup generated by $G_0$ and $W(G)$ is naturally isomorphic to the direct product $G_0\times W(G)$ and is open of finite index. 
\end{lem}
\begin{proof}
As $W(G_0)=1$, we have $W(G)\cap G_0=\{1\}$, hence $W(G)$ and $G_0$ centralize each other. Since $W(G)$ is compact, this is enough to ensure that the natural homomorphism $G_0\times W(G)\to G_0W(G)$ is a topological isomorphism onto a closed subgroup.

Since $\pi_0(G)$ is compact, $G/W(G)$ is a Lie group with finitely many components. Now $G/G_0W(G)$ is both a quotient of $G/W(G)$ and of $G/G_0$, which is totally discontinuous, and therefore $G/G_0W(G)$ is finite.
\end{proof}

Under the same assumption, we need to characterize $W(G)$ without referring to any normality assumption (since in the sequel $G$ will vary among open subgroups of a given group).

\begin{lem}\label{lce}
Let $G$ be a Lie group with $\pi_0(G)$ finite and $W(G)=1$. Then $G$ has no non-trivial compact subgroup centralizing $G_0$. 
\end{lem}
\begin{proof}
Let $K$ be a compact subgroup centralizing $G_0$; let us assume that $K$ is maximal for this property. Then $K\cap G_0$ is compact and central in $G_0$ so is trivial. In particular, $K$ is finite. Let $S$ be a maximal compact subgroup of $G$ containing $K$, and let $N$ be the centralizer of $G_0$ in $G$. Since $G=SG_0$ \cite[Theorem~3.1]{Mo}, any conjugate of $K$ is contained in $S$. Moreover since $K$ is contained in $N$ which is normal, any conjugate of $K$ is contained in $S\cap N$. Now $S\cap N$ is a compact subgroup centralizing $G_0$, so by maximality of $K$, we have $K=S\cap N$ and therefore $K$ contains all its conjugates, so is normal in $G$. Since $W(G)=1$, this implies $K=1$. 
\end{proof}

\begin{lem}\label{inw}
Let $G$ be a locally compact group with $\pi_0(G)$ compact and $W(G_0)=1$. Then any compact subgroup of $G$ centralizing $G_0$ is contained in $W(G)$.
\end{lem}
\begin{proof}
Let $K$ be a compact subgroup with $[K,G_0]=1$. By Lemma \ref{lce}, the image of $K$ in $G/W(G)$ is trivial.
\end{proof}

Therefore, when $\pi_0(G)$ is compact and $W(G_0)=1$ holds, $W(G)$ appears as the maximal compact subgroup centralizing $G_0$.

\begin{thm}\label{thm_decompo}
Let $G$ be a locally compact group with $\pi_0(G)$ elliptic and $W(G_0)=1$. Then $G$ has a unique maximal subgroup $W'(G)$ among those closed elliptic subgroups centralizing $G_0$. Moreover, the natural map $G_0\times W'(G)\to G_0W'(G)$ is a topological isomorphism onto an open subgroup of $G$.
\end{thm}
\begin{proof}
Write $G$ as the union of a net of open subgroups $G_i$ with $\pi_0(G_i)$ compact. If $G_i\subset G_j$, then 
\begin{equation}\label{eee}
W(G_j)\cap G_i= W(G_i)\end{equation}
Indeed, $W(G_j)\cap G_i\subset W(G_i)$ by definition of $W$, and the reverse inclusion follows from Lemma \ref{inw} applied inside $G_j$. Set $W'(G)=\bigcup W(G_i)$. From (\ref{eee}) it follows that $W'(G)\cap G_i=W(G_i)$ for all $i$, and in particular $W'(G)$ is closed. Clearly it is elliptic and centralizes $G_0$. Conversely if $H$ is closed, elliptic and centralizes $G_0$, then for all $i$ $H\cap G_i$ is compact and therefore by Lemma \ref{inw} is contained in $W(G_i)$, hence in $W'(G)$, so that $H$ is contained in $W'(G)$.

To check that the natural map $G_0\times W'(G)\to G$ is a topological isomorphism onto its image, it is enough to check that it is a proper injective homomorphism. It is injective since $G_0\cap W'(G)=1$. To check properness, let $((g_n,w_n))$ be a sequence with $g_n\in G_0$, $w_n\in W'(G)$, and $g_nw_n\to 1$. Then since $G_i$ is open, eventually $g_nw_n\in G_i$, so $w_n\in W(G_i)$, and therefore properness is reduced to properness in restriction to $G_0\times W(G_i)$, which was established in Lemma \ref{cod}.

Finally $G_0W'(G)$ contains the open subgroup $G_0W(G_i)$ of the open subgroup $G_i$, so is open itself in $G$.
\end{proof}

\begin{cor}\label{confi}
Let $G$ be a locally compact group with $W(G_0)=1$, and assume that $G$ has a contraction. Then $G$ has a characteristic subgroup $W'(G)$ satisfying the properties stated in Theorem \ref{thm_decompo}, and $G_0W'(G)$ has finite index.
\end{cor}
\begin{proof}
By Lemma \ref{cell}, $G/G_0$ is elliptic and Theorem \ref{thm_decompo} applies. It is obvious from its statement that the subgroup $W'(G)$ is characteristic, and in particular the open subgroup $G_0W'(G)$ is characteristic as well, so is stable under the given contraction. Therefore the discrete group $G/G_0W'(G)$ also has a contraction, so must be finite.
\end{proof}

This contains \cite[Proposition 4.2]{Sie} as a particular case.

\begin{cor}
Let $G$ be a locally compact group and assume that $G$ has a strict contraction. Then $G$ decomposes (canonically) as a product of characteristic subgroups $G=G_0\times W'(G)$.
\end{cor}
\begin{proof}
The strict contraction $\alpha$ strictly contracts the compact normal subgroup $W(G_0)$ and therefore $W(G_0)=1$, so Corollary \ref{confi} applies. Moreover $\alpha$ induces a strict contraction of the finite group $G/(G_0W'(G))$, so this is the trivial group, that is, $G=G_0W'(G)$.
\end{proof}

\section{Quasi-isometry invariance}\label{qii}

The invariance under quasi-isometries of the $L^p$-cohomology was obtained in all degrees in \cite{Pan95}. As that preprint from 1995 is not yet published, and the case of degree one, used in the proof of the vanishing results of the $L^p$-cohomology in Theorem \ref{main}, is considerably simpler than in higher degree, we include a full proof.

The following coarse notion of (first) $L^p$-cohomology in degree one is
essentially due to \cite{Pan95} (see also the chapter about
$L^p$-cohomology in \cite[Chap.~8]{Gro}, and \cite{T1} for the case of degree 1). Here, we use the notation of \cite{T1}, but we extend the definition to spaces which are not necessarily 1-geodesic. This will allow us to apply our result to non-compactly generated locally compact groups. The fact that this definition is equivalent to the one we gave in introduction follows from the proof of \cite[Theorem 5.1]{T1}. 

Let $X=(X,d,\mu)$ be a metric measure space, and let $p\geq 1$.
For every $s>0$, we write $\Delta_s=\{(x,y)\in X^2, d(x,y)\leq s\}.$

First, let us introduce the $p$-Dirichlet space $D_p(X)$.
\begin{itemize}
\item The space $D_p(X)$ is the set of measurable functions $f$ on
$X$ such that
$$\int_{\Delta_s}|f(x)-f(y)|^pd\mu(x)d\mu(y)<\infty$$
for every $s>0$.
\item Let us equip $D_p(X)$ with
the topology induced by the following semi-norms (for all ${s>0}$)  
$$\|f\|_{D_p,s}=\left(\int_{\Delta_s}|f(x)-f(y)|^pd\mu(x)d\mu(y)\right)^{1/p}.$$

\end{itemize}
\begin{defn}
The first $L^p$-cohomology of $X$ is the space
$$H^1_p(X)=D_p(X)/(L^p(X)+\mathbf{R}),$$ and the first reduced $L^p$-cohomology
of $X$ is the space
$$\overline{H^1}_p(X)=D_p(X)/\overline{(L^p(X)+\mathbf{R})}^{D_p(X)}.$$
\end{defn}

Let $X= (X,d,\mu)$ be a metric measure space satisfying the following ``bounded geometry" condition:
for all $x\in X$, and $r>0$, 
$$v(r)\leq \mu(B_X(x,r))\leq V(r),$$ where $v$ and $V$ are increasing positive functions on $(0,\infty).$ It is not difficult to check that this condition is automatically satisfied if the group of measure-preserving isometries of $X$ acts cocompactly. In particular, a locally compact group equipped with a left-invariant, proper metric has bounded geometry (properness is assumed here only  for the balls to have finite volume).

We will use the notation $f\preceq_T g$ when $f\leq Cg$ for some constant $C$ depending ``only" on $T$.

\begin{thm}
Suppose that two metric measure spaces $X$ and $Y$ with bounded geometry are coarse equivalent. Then there exists a topological isomorphism between their first $L^p$-cohomology. 
\end{thm}

\begin{proof} First, note that if the spaces are discrete  (equipped with the counting  measure) and if  $\varphi$ is a {\it bijective} coarse equivalence from $X$ to $Y$, then $\varphi$ induces a topological isomorphism on the first $L^p$-cohomology. Therefore, owing to the following lemma,  we can suppose that $Y$ is a {\it discretization} of $X$, i.e. a sub-metric space of $X$ such that $d(y,y')\geq 1$ for all $y,y'\in Y$, and such that the inclusion $Y\subset X$ is a coarse equivalence.  This is equivalent to the existence of some $R>0$  such that $X\subset \bigcup_{y\in Y} B_X(y,R).$ Moreover, since $X$ has bounded geometry, at most $N=N(R)$ of these balls intersect.

\begin{lem}
Let $\varphi:\;X\to Y$ be a coarse equivalence. There exists a discretization $X_d$ of $X$, such that the restriction of $\varphi$ to $X_d$ is one to one, and such that $\varphi(X_d)$ is a discretization of $Y$.
\end{lem}
\begin{proof}
By definition of a coarse equivalence, there exists $R>0$ such that points at distance at least $R$ in $X$ are mapped to points at distance at least $1$ in $Y$. Hence the lemma follows by taking for $X_d$ a  maximal $R$-separated net in $X$.
\end{proof}

The two main ingredients of the proof, that we will use thoroughly without mentioning them, are
H\"older's inequality and  the fact that $X$ and $Y$ have bounded geometry.

For every $T>0$, let us define an operator $P_T$ on $L^p(X)$ as follows:
$$P_Tf(x_0)= \mathbf{E}_{B_X(x_0,T)}f,$$
where $\mathbf{E}_Af:=\frac{1}{\mu(A)}\int_A f(x)d\mu(x).$ 
Also, we denote $\nabla f(x,x'):= f(x)-f(x').$

\begin{lem}
For every $T\geq 0$, the linear  map $\psi_T(f)(y)=P_Tf(y)$, from $D_p(X)$ to $D_p(Y)$, induces a continuous map on the $L^p$-cohomology. 
\end{lem}
\begin{proof}  Note that $\psi_T$ obviously preserves constant functions and it is easy to see that it is bounded on $L^p$. Let $s\geq T.$ We will sketch the proof that $\psi_T$ is continuous from $D_p(X)$ to $D_p(Y)$ (the details are straightforward and therefore left to the reader).
\begin{eqnarray*}
\sum_{d(z,z')\leq s}|P_Tf(z)-P_Tf(z')|^p &\preceq_s & \sum_{d(z,z')\leq s}\left(\mathbf{E}_{B_X(z,2s)\times B_X(z,2s)}|\nabla f |\right)^p\\
& \leq &\sum_{d(z,z')\leq s}\mathbf{E}_{B_X(z,2s)\times B_X(z,2s)}(|\nabla f |^p)\\
& \preceq_s  &\int_{d(x,x')\leq 5s}|f(x)-f(x') |^pd\mu(x)d\mu(x').\qedhere
\end{eqnarray*}
 \end{proof}
 
Now, let us choose $T\geq R$, so that $X\subset \bigcup_{y\in Y} B_X(y,T).$ Let $$S_T(x)=\sum_{z\in Y}1_{B_X(z,T)}, \; \forall x\in X,$$
which satisfies $1\leq S_T\leq N(T)<\infty$, and for all $z\in Y$,
$$S_T^z=1_{B_X(z,T)}/S_T.$$
Note that $(S_T^z)$ forms a partition of unity for $X$.

\begin{lem}
For every $T\geq R$, the linear map $\phi_T(f)=\sum_{z\in Y} f(z)S_T^z$, from $D_p(Y)$ to $D_p(X)$, induces a continuous map on the $L^p$-cohomology. 
\end{lem}
 
\begin{proof} It is clear that $\phi_T$ preserves constant functions, and the proof that it is bounded on $L^p$ is easy and left to the reader. Let us sketch the proof that $\phi_T$ is continuous on $D_p$.  Let $(x,x')\in \Delta_s(X)$ for some $s\geq T.$
\begin{eqnarray*}
|\phi_T(f)(x)-\phi_T(f)(x')|^p & = & |\sum_{z,z'\in Y}(f(z)-f(z'))S^z_T(x)S^{z'}_T(x')|^p\\
& = &  |\sum_{z,z'\in B_X(x, 3s)}(f(z)-f(z'))S^z_T(x)S^{z'}_T(x')|^p\\
& \preceq_s &  \sum_{z,z'\in B_X(x, 3s)}|f(z)-f(z')|^p. 
\end{eqnarray*}
Accordingly
\begin{eqnarray*}
\int_{d(x,x')\leq s}|\phi_T(f)(x)-\phi_T(f)(x')|^p & \preceq_s & \sum_{d(z,z')\leq 6s}|f(z)-f(z')|^p.\qedhere\end{eqnarray*}
\end{proof}

\begin{lem}
For every $T\geq R$, the maps $\phi_T\circ \psi_T$, and $\psi_T\circ\phi_T$ induce the identity on the $L^p$-cohomology of respectively $X$ and $Y$. 
\end{lem}
For every $x_0\in X$, we have  
\begin{eqnarray*}
|f(x_0)-\phi_T\circ\psi_T(f)(x_0)|^p &= & |\sum_{z\in Y}[f(x_0)-P_Tf(z)]S_T^z(x_0)|^p\\
& = & |\sum_{z\in Y}\mathbf{E}_{B_X(z,T)}[f(x_0)-f]S_T^z(x_0)|^p\\
& = & |\sum_{d(z,x_0)\leq T}\mathbf{E}_{B_X(z,T)}[f(x_0)-f]S_T^z(x_0)|^p\\
&\preceq_T&  \sum_{d(z,x_0)\leq T}|\mathbf{E}_{B_X(z,T)}[f(x_0)-f]|^p\\
&\leq&  \sum_{d(z,x_0)\leq T}\mathbf{E}_{B_X(z,T)}|f(x_0)-f|^p\\
& \preceq_T & \mathbf{E}_{B_X(x_0,2T)}|f(x_0)-f|^p.
\end{eqnarray*}
It follows that $$\|f-\phi_T\circ\psi_T(f)\|_p\preceq_T \|f\|_{D_{p,2T}},$$ 
and therefore $\phi_T\circ\psi_T$ induces the identity on $H^1_p(X).$

For every $z_0\in Y$, 
\begin{eqnarray*}
|f(z_0)-\psi_T\circ\phi_T(f)(z_0)|^p &= & |P_T[f(z_0)-\phi_T(f)](z_0)|^p\\
&=& |P_T[\sum_{z\in Y}(f(z_0)-f(z))S_T^z](z_0)|^p\\
&=& |\sum_{z\in Y}(f(z_0)-f(z))\mathbf{E}_{B_X(z_0,T)}S^z_T|^p\\
&\preceq_T& \sum_{d(z,z_0)\leq 2T}|f(z_0)-f(z))|^p.
\end{eqnarray*}
We used the fact that $\mathbf{E}_{B_X(z_0,T)}S^z_T\preceq_T 1_{B(z_0,2R)}(z).$
We deduce from the above that
$$\|f-\psi_T\circ\phi_T(f)\|_p\preceq_T \|f\|_{D_{p,2T}}.$$
Which implies that $\psi_T\circ\phi_T$ induces the identity on $H^1_p(Y).$
 \end{proof}

\end{document}